\documentclass[12pt,reqno]{amsart}
\usepackage{amsfonts,amsmath,amscd,amssymb,amsthm,color,graphicx,wrapfig,url}
\usepackage{enumerate,hyperref,perpage,url,dsfont}
\usepackage[margin=2.2cm, a4paper]{geometry}

\theoremstyle{plain}
\newtheorem{thm}{Theorem}

\newtheorem{cor}[thm]{Corollary}
\newtheorem{lemma}[thm]{Lemma}
\newtheorem{prop}[thm]{Proposition}

\theoremstyle{remark}

\newtheorem*{rmk}{\textbf{Remark}}

\numberwithin{equation}{section}
\numberwithin{thm}{section}

\MakePerPage{footnote} \allowdisplaybreaks 
\newcommand{\1}{\ensuremath{\mathds{1}}}

\newcommand{\norm}[1]{\left|\!\left|{#1}\right|\!\right|}

\newcommand{\dist}{\mathrm{dist}}

\newcommand\E{\mathbb E}

\newcommand\R{\mathbb R}
\renewcommand{\S}{\mathbb S}

\newcommand{\Trace}{\mathrm{Tr}\,}

\newcommand{\Vol}{\mathrm{Vol}}

\renewcommand{\j}{\boldsymbol{j}}

\newcommand{\FT}{\mathcal{F}}

\title[Filament structure of random waves ]{Filament structure of random waves }

\author{Melissa Tacy}
\address{Department of Mathematics, The University of Auckland, Auckland, New Zealand}
\email{melissa.tacy@auckland.ac.nz}

\subjclass[2020]{58J50, 35S05, 35P20, 60B10}

\keywords{Eigenfunction equidistribution at small scales, random waves, filament structure, scaring, phase space}

\date{}

\begin{document}
\maketitle

\begin{abstract}
We investigate the small-scale equidistribution properties of random waves, $u$, in $\R^{n}$. Numerical evidence suggests that such objects display a fine scale filament structure. We show that the X-ray of $|u|^{2}$ along any line segment is uniformly equidistributed. So any limiting behaviour must be weaker than $L^{2}$ scaring. On the other hand, we show that   at Planck scale in phase space there are (with high probability) logarithmic fluctuations above what would be expected given equidistribution. Taken together these results suggest that the filament structure may be a configuration space echo of the phase-space concentrations.  \end{abstract}

In this paper we are concerned with the small-scale structure of random plane waves. Random plane waves are functions $u$ of the following form:
$$u=\sum_{\xi_{j}\in\Lambda}c_{j}e^{\frac{1}{h} \langle x,\xi_{j}\rangle}$$
where $\Lambda\subset \R^{n}$ is a set of equally spaced momenta. Exact choices for $\Lambda$ vary somewhat between papers in this field but $\Lambda$ is usually a neighbourhood of the unit sphere $\S^{n-1}$. The randomness is injected through the coefficients $c_{j}$ which are chosen according to a probability distribution. Such functions $u$ were conjectured, by Berry \cite{Ber77}, to provide a good model for the behaviour of chaotic modes in billiard systems. 

There are a number of interesting probability distributions from which to draw the $c_{j}$. For example we may choose to look at cases where each $c_{j}$ is chosen as an independent random variable such as Gaussian or Rademacher (see for instance  \cite{Ber77}, \cite{Zel09},\cite{dCI20}). Alternatively in situations where it is preferable to be able to fix the $\ell^{2}$ norm of $c=[c_{1},\dots,c_{|\Lambda|}]^{T}$ the coefficients may be chosen from a uniform probability density on the high dimensional sphere $\S^{|\Lambda|-1}$ (see for example \cite{Zel96},\cite{BL13},\cite{Map13},\cite{Zel14} and \cite{H17}).

 In this paper we treat a model where each $c_{j}$ is a (independent, identically distributed) Gaussian random variable and $\Lambda=\Lambda_{\beta}$ is a set of $h$-separated momenta drawn from $\S^{n-1}\times[1-h^{\beta},1+h^{\beta}]$ for $0\leq\beta\leq 1$. This range of $\beta$ corresponds to the range of random waves for which we have numerical  experiments. The website of Alex Barnett \cite{webB} records a number of these experiments including videos generated at the AIM workshop \emph{Topological complexity of random sets} showing the effect of varying $\beta$. It is convenient to choose $\Lambda_{\beta}$ so that the momenta are equally spaced, this is however not strictly necessary. What is necessary is that the spacing between momenta is never less than $h$ (this is an uncertainty principle requirement) and that the number of momenta in a region scales with the volume of that region.  To allow normalisation we restrict our attention to the behaviour of $u$ inside the ball of radius one about zero, $B_{1}(0)$. Then the variance for the Gaussian random variables is  determined by adopting the convention that
$$\E\left[\norm{u}^{2}_{L^{2}(B_{1}(0))}\right]=\Vol(B_{1}(0)).$$
This then means that the variance $\sigma^{2}=\frac{1}{ N}$ where $N=|\Lambda|$. 

The key questions of quantum chaos (see for example the survey \cite{Zel19}) are concerned with limits of $\langle P u_{h},u_{h}\rangle$ where $P$ is a semiclassical pseudodifferential operator. A sequence of $u_{h}$ on a compact manifold $M$ with $h\to 0$ is said to equidistribute if
$$\langle P u_{h},u_{h}\rangle \to  \frac{1}{\Vol(S^{\star}M)}\int_{S^{\star}M}p\,d{L}$$
where $S^{\star}M$ is the unit co-tangent bundle equipped with Liouville measure and $p$ is the principal symbol of $P$. This leads to a natural question, are random waves equidistributed? 

It is reasonably easy to ascertain that if $p(x,\xi)$ (the symbol of $P$) is independent of $h$ then random waves are indeed equidistributed. A more subtle question pertains to the small-scale structure. If for example $p(x,\xi;h)$ is a symbol that is zero off a small (shrinking in $h$) region of space in $T^{\star}M$ does equidistribution (in an almost sure sense) still hold? If not does a weaker form of equidistribution hold where   $\langle P u,u\rangle$ is only  proportional to the phase-space volume of $p(x,\xi;h)?$ 

\begin{wrapfigure}{r}{0.5\textwidth}
\includegraphics[scale=0.4]{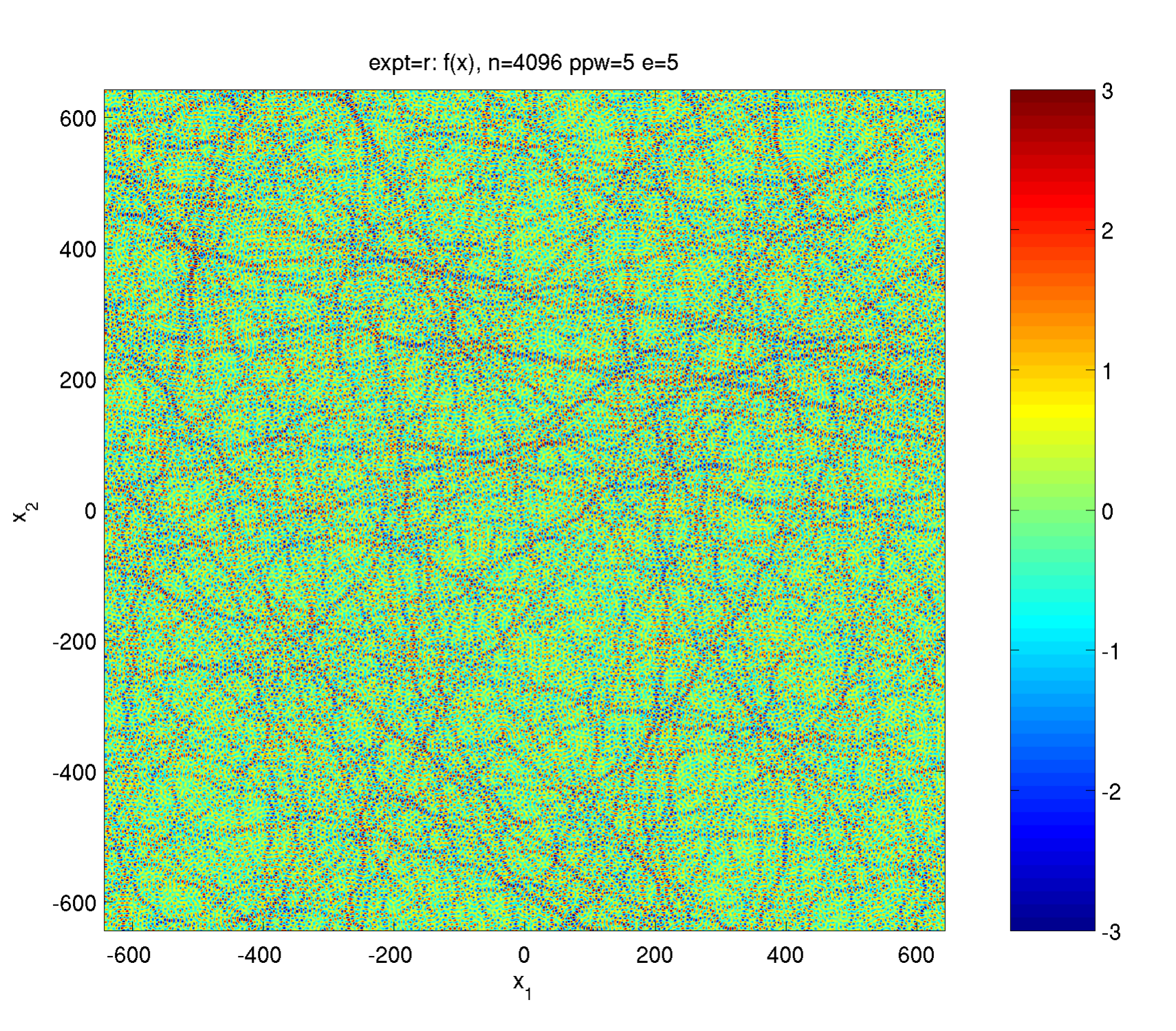}
\caption{A random plane waves in dimension two. Reproduced from \cite{webB} with permission of Alex Barnett.}\label{fig:rw}
\end{wrapfigure}
\leavevmode 
Numerical studies of random waves have suggested that at very small scales there are fluctuations that exceed equidistribution. As seen in Figure \ref{fig:rw} (reproduced with permission of Alex Barnett) random waves appear to have enhancements along some straight lines. This so named filament structure was intensely discussed at the AIM workshop \emph{Topological complexity of random sets} (see the report \cite{AIM09}). Some participants taking the opinion that the apparent structure was simply a numerical artefact other believing that it could be quantified. No firm conclusion was reached at the workshop or since.   In this paper we take some steps toward understanding what precisely gives rise to this filament structure. As observed in the AIM report \cite{AIM09} the linear structure appears to be constructed of many smaller filaments roughly aligned along a straight lines. The simplest form of failure to equidistribute would then be to have some straight line segments $\gamma$ where $\norm{u}_{L^{2}}$ was unusually large. Due to VanderKam \cite{VK97} and Zelditch \cite{Zel96} the maximum fluctuation we could expect would be logarithmic. In this paper we will see that this sort of concentration does not occur. In Theorems \ref{thm:Xrayexp} and \ref{thm:Xrayunif} we  show that the X-ray transform of the modulus squared of a random wave is uniformly equidistributed with high probability. That is for $\gamma_{(x,\xi)}$ a unit length line segment, properly contained in $B_{1}(0)$, from $x$ in direction $\xi$ the random variable $\norm{u}_{L^{2}(\gamma_{(x,\xi)})}=\left(\int_{\gamma_{(x,\xi)}}|u|^{2}dl_{\gamma}\right)^{1/2}$ obeys
$$\E\left[\norm{u}_{L^{2}(\gamma_{(x,\xi)})}\right]=1+o(1)\quad h\to 0$$
and that there is a $\kappa>0$ (details in Theorem \ref{thm:Xrayunif}) so that
\begin{align}Pr\left\{c: \exists (x,\xi), \left|\norm{u}_{L^{2}(\gamma_{(x,\xi)})}-1\right|>m(h)\right\}&\leq \exp\left(-CN^{2\kappa}m^{2}(h)\right)\label{unifXray}\\
&m(h)>N^{-\kappa+\epsilon},m(h)\to 0.\nonumber\end{align}
In terms of what we might expect to see in a numerical simulation the uniformity (in $(x,\xi)$) of the tail bound \eqref{unifXray} is crucial. If for example we were only able to show that for fixed $(x,\xi)$
$$Pr\left\{c:  \left|\norm{u}_{L^{2}(\gamma_{(x,\xi)})}-1\right|>m(h)\right\}\leq \exp\left(-CN^{2\kappa}m^{2}(h)\right)\quad m(h)>N^{-\kappa+\epsilon},m(h)\to 0$$
 we could conclude that if we picked a line segment then considered all potential random waves the chance that we would see a failure to equidistribute on that particular segment is small. For an example of this kind of result see \cite{E16} (in the setting where the coefficients are chosen according to a uniform distribution on a high dimensional sphere). Such a statement would not preclude the possibility that most random waves have some $(x,\xi)$ for which  $\norm{u}_{L^{2}(\gamma(x,\xi))}$ is large. Therefore to rule out structure of this form we do indeed need the stronger bound \eqref{unifXray}. In Section \ref{sec:radon} we establish Theorem \ref{thm:main1} which guarantees uniform equidistribution of $\norm{u}_{L^{2}(\gamma(x,\xi))}$.

\begin{thm}\label{thm:main1}
Let $F_{\beta}(x,\xi)$ be the random variable given by
$$F_{\beta}(x,\xi)=F_{\beta}(x,\xi;c)=\norm{u}_{L^{2}(\gamma(x,\xi))}=\norm{\sum_{\lambda_{j}\in\Lambda_{\beta}}c_{j}e^{\frac{i}{h}\langle \cdot,\xi_{j}\rangle}}_{L^{2}(\gamma(x,\xi))}$$
then
$$\E[F_{\beta}(x,\xi)]=1+o(1)\quad h\to 0.$$
Further there is a  $\kappa(n)$ (explicitly given by in Section \ref{sec:radon} by \eqref{kappan}) such that if $m:\R^{+}\to\R^{+}$, $m(h)\to 0$ as $h\to 0$,
$m(h)\geq{}N^{-\kappa(n)+\epsilon}$ and  the exception set $S(m)$ is given by
$$S(m)=\{c:\exists (x,\xi) \text{ so that } |F_{\beta}(x,\xi;c)-1|\geq{}m(h)\}$$
then
\begin{equation}Pr(S(m))\leq \exp\left(-N^{2\kappa(n)}m^{2}(h)\right).\label{eqn:exest}\end{equation}
\end{thm}

\begin{rmk} We will see, in Section \ref{sec:radon}, that in fact
$$\E[F_{\beta}(x,\xi)^{2}]=1$$
independent of $h$. The error is introduced by approximating $\E[F_{\beta}(x,\xi)]$ by $\left(\E[F_{\beta}(x,\xi)^{2}]\right)^{1/2}$. The validity of the approximation follows from a concentration of measure argument which in turn explicitly controls the quantity.
$$\left|\E[F_{\beta}(x,\xi)]-\left(\E[F_{\beta}(x,\xi)^{2}]\right)^{1/2}\right|.$$
The explicit expressions for this control are found in Theorem \ref{thm:Xrayexp} (equations \eqref{eqn:expecF2}, \eqref{eqn:expecF3} and \eqref{eqn:expecFn}).
\end{rmk}

It is instructive to compare Theorem \ref{thm:main1} with general results about restriction of spectral clusters to curves. When $\beta=1$ the requirement that $\xi_{j}\in\Lambda_{\beta}$ ensures that $u$ is a spectral cluster of width one. Therefore we can apply the results of Burq-G\'{e}rard-Tvetkov \cite{BGT} and Hu \cite{Hu09} to obtain, that for any choice of coefficients,
\begin{align*}
\norm{u}_{L^{2}(\gamma(x,\xi))}&\lesssim h^{-\frac{1}{4}}\norm{u}_{L^{2}(B_{2}(0))}\\
\norm{u}_{L^{2}(\mathcal{C})}&\lesssim h^{-\frac{1}{6}}\norm{u}_{L^{2}(B_{2}(0))}\quad\text{$\mathcal{C}$ geodesically curved}.\end{align*}
The results of Theorem \ref{thm:main1} should be interpreted as saying that choices of coefficients $c=[c_{1},\dots,c_{N}]$ that saturate the general bounds are highly unusual. 

In the setting of $(M,g)$ a Riemannian manifold with ergodic geodesic flow the quantum ergodic restriction results of \cite{TZ13} and \cite{DZ13} imply that for density one subsequences of exact eigenfunctions $\phi_{j}$,
$$\norm{\phi_{j}}_{L^{2}(\gamma)}\to L(\gamma)$$
where $L(\gamma)$ is the length of the curve $\gamma$. The agreement between this and Theorem \ref{thm:main1} should be seen as evidence of the agreement between quantum ergodic systems and the predictions of the random wave model. 

We then turn our attention to the phase-space picture and consider a random variable $G_{\alpha,\beta,\mu}(x,\xi)$ which measures the phase-space concentration of $u$ near the point $(x,\xi)\in\R^{n}\times \S^{n-1}$ given by
\begin{equation}G_{\alpha,\beta,\mu}(x,\xi)=G_{\alpha,\beta,\mu}(x,\xi;c)=\norm{P^{\alpha,\mu}_{(x,\xi)}u(\cdot)}_{L^{2}}.\label{intro:Gdef}\end{equation}
Here $P^{\alpha,\mu}_{(x,\xi)}$ is a semiclassical pseudodifferential operator that localises $u$ in phase space near $(x,\xi)$. The parameter $\alpha$ controls the shape of the localisation, the parameter $\mu$ controls the distance from Planck scale. When $\mu=1$ the operator $P^{\alpha,\mu}_{(x,\xi)}$ localises $u$ up to the limits allowed by the uncertainty principle. As $\mu$ increases the localisation becomes more relaxed. Normalisations are chosen so that if $u$ is phase-space equidistributed $|G_{\alpha,\beta,\mu}(x,\xi)|=1+O(\mu^{-1})$. The details of symbol of $P^{\alpha,\mu}$ are described in Section \ref{sec:phasespace}. We discover that once $\mu$ is large, that is $\mu\approx h^{-\epsilon})$, we again obtain a uniform equidistribution result similar to that for the X-ray transform of $|u|^{2}$. However when $\mu\leq C$ we not only fail to obtain equidistribution but are able to show that with high probability there are logarithmic enhancements near some points $(x,\xi)$.  The intermediate scales $C\leq \mu\leq h^{-\epsilon}$  remain something of a mystery and will likely require very delicate analysis to fully resolve.

\begin{thm}\label{thm:mconcenphase}
Let $G_{\alpha,\beta,\mu}(x,\xi)$ be given by \eqref{intro:Gdef}. For $\mu\leq C$ and small enough (but independent of $h$)  there exist constants $C_{1},C_{2}$ so that
\begin{equation}C_{1}\sqrt{\log(1/h)}\leq \E\left[\norm{G_{\alpha,\beta,\mu}}_{L^{\infty}(B_{1}(0)\times \S^{n-1})}\right]\leq C_{2}\sqrt{\log(1/h)}.\label{intro:Pexpectinfty}\end{equation}
On the other hand if $\mu\geq h^{-\epsilon}$, then for $(x,\xi)\in B_{1}(0)\times \S^{n-1}$,
\begin{equation}
\E\left[G_{\alpha,\beta,\mu}(x,\xi)\right]=1+O\left(h^{\frac{\epsilon}{2}}\right).\label{intro:expectinfty}\end{equation}
Further if $\delta>0$ and $m(h)\geq{}\max\left(h^{\epsilon-\delta}, h^{\frac{2\epsilon(n+1)}{4(n-\beta)}-\delta}\right)$,
\begin{equation}Pr\{c:\exists(x,\xi),\text{ so that } \left|G_{\alpha,\beta,\mu}(x,\xi;c)-1\right|>m(h)\}\leq \exp(-N^{\frac{2\epsilon(n+1)}{4(n-\beta)}}m^{2}(h)).\label{eqn:Guniform}\end{equation}

\end{thm}

\begin{rmk}
The proof of Theorem \ref{thm:mconcenphase} depends on estimating $\norm{G_{\alpha,\beta,\mu}(\cdot,\cdot)}_{L^{\infty}}$ by an $L^{p_{h}}$ norm where $p_{h}$ is finite but growing in $h$. Then $\norm{G_{\alpha,\beta,\mu}(\cdot,\cdot)}^{p_{h}}_{p_{h}}$ is directly computed. It is the lower bound on this norm that requires $\mu\leq C$ (see equation \eqref{Tracelower}). The key technical dependancies for $\mu$ are determined in Lemma \ref{lem:tracebnds}, in particular the comparison between \eqref{Tracebnds} and \eqref{Tracesqbnd}. 
\end{rmk}

Therefore once we are at Planck scale we expect to see some regions of phase space that where $u$ displays concentration that is logarithmically larger than that predicted by equidistribution. This result should be seen as aligning with Burq-Lebeau's result \cite{BL13} that 
\begin{equation}C_{1}\sqrt{\log(1/h)}\leq \E\left[\norm{u}_{L^{\infty}}\right]\leq C_{2}\sqrt{\log(1/h)}.\label{BL}\end{equation}
Indeed when $\alpha=0$ the operator $P^{\alpha,\mu}_{(x,\xi)}$ localises $u$ to a small ball of radius $h$ around $x$. Since $u$ is constituted of waves of frequency of size $1/h$ the value of $|u(x)|$ cannot vary much on balls of radius $h$. Taking into account the normalisation we therefore get that
$$\norm{P^{0,\mu}_{(x,\xi)}}_{L^{2}}^{2}\approx|u(x)|^{2}.$$
Therefore the results of this paper for $\alpha=0$ should be seen as morally equivalent to the Burq-Lebeau result. 

The question then remains, just what are we seeing in numerical simulation that display filament structure? We know from Theorem \ref{thm:main1} any scaring along a line segment $\gamma(x,\xi)$ cannot be enough to register in the $L^{2}$ norm. However it is quite possible to have large number of logarithmically bright points aligned along a ray without the $L^{2}$ norm along that ray becoming large. To the human eye examining a numerical experiment, such a situation may appear as a scar. This suggests we should look at weaker forms of structure rather than $L^{2}$ scaring. One possibility, suggested by the phase-space failure to equidistribute, is that the bright points tend to align along rays that have a Planck-scale enhancement in the phase-space sense.

Throughout this paper we will use concentration of measure arguments to allow us to ``commute the expectation operator with a function $\phi$''. The idea is that if $\phi:[0,\infty)\to[0,\infty)$  is invertible, we will use measure concentration to find (see Section \ref{sec:mc}) sufficient conditions so that 
$$C_{1}\phi^{-1}\left(\E[\phi(X)]\right)\leq \E\left[X\right]\leq C_{2}\phi^{-1}\left(\E[\phi(X)]\right)\quad C_{1},C_{2}\in \R^{+}.$$
Here $X$ is a random variable that takes values in $[0,\infty)$. The rational for seeking this kind of estimate is to be able to chose $\phi$ so that $\E[\phi(X)]$ is easy (or at least tractable) to compute. Just by assuming $\phi$ is convex and strictly increasing we can get the upper bound (by Jensen). However, typically, we need the full inequality. In some cases we will in fact be able to find lower order control on $\left|\phi\left(\E[X]\right)-\E\left[\phi(X)\right]\right|$ therefore obtaining control on the ``Jensen gap''.

\section{Measure concentration}\label{sec:mc}

In this section we set up those measure concentration results relevant to this paper. We use those results to find sufficient conditions (on $X$ and $\phi$) so that
$$C_{1}\phi^{-1}\left(\E[\phi(X)]\right)\leq \E\left[X\right]\leq C_{2}\phi^{-1}\left(\E[\phi(X)]\right)\quad C_{1},C_{2}\in \R^{+}$$
or
$$\E\left[X\right]=\phi^{-1}\left(\E[\phi(X)]\right)+Er$$
where $Er$ is small compared to $\phi^{-1}\left(\E[\phi(X)]\right)$. The major measure concentration result we use here is analogous to the Levy concentration of measure for coefficients chosen uniformly on a high dimensional sphere. In that case the key driver of the proof was the isoperimetric inequality on spheres, which states that of all sets of fixed measure the ones with the ``smallest perimeter'' are hemispheres. To deal with Gaussian variables we simply replace this inequality with the Gaussian isoperimetric inequality, see \cite{SuCir74} (Theorem 1)  and \cite{Borell75} (Theorem 3.1), which states that of all sets of fixed measure the ones with ``smallest perimeter'' are half-planes. 

\begin{lemma}\label{lem:measureconc}
Let $c=[c_{1},\dots,c_{N}]^{T}$ be an $N$-tuple of independent Gaussian variables each with variance $\sigma^{2}=N^{-1}$.  Let $\mu$ be the joint $N$-dimensional Gaussian measure
\begin{equation}\mu=\frac{N^{N/2}}{(2\pi)^{N/2}}e^{-\frac{N|c|^{2}}{2}}\lambda^{N}\label{Gaussianmeasure}\end{equation}
where $\lambda^{N}$ is the standard $N$-dimensional Lebesgue measure. Suppose $G=G(c)$ is a random variable obeying the Lipschitz bounds
\begin{equation}|G(c)-G(d)|\leq L\norm{c-d}_{\ell^{2}}\label{Lipest}\end{equation}
Then if $M_{G}$ is a median for $G$ there is a constant $C$ so that 
\begin{equation}Pr\{G\geq M_{G}+t\}\leq Ce^{-\frac{Nt^{2}}{2L^{2}}}\label{concen1}\end{equation}
and
\begin{equation}Pr\{G\leq M_{G}-t\}\leq Ce^{-\frac{Nt^{2}}{2L^{2}}}.\label{concen2}\end{equation}
\end{lemma}

\begin{proof}
Let the set $B\subset\R^{N}$ be given by
$$B=\{c\in\R^{N}:G(c)\leq M_{G}\}.$$
Since $G$ obeys the Lipschitz estimate
$$|G(c)-G(d)|\leq L\norm{c-d}_{\ell^{2}}$$
$$B_{t/L}= \left\{ c\in\R^{N}: \dist(c,B)<\frac{t}{L}\right\}\subset \{c\in \R^{N}\mid G(c)< M_{G}+t\}. $$
Then the Gaussian isoperimetric inequality  (see \cite{SuCir74} and \cite{Borell75})  tells us that
$$\mu(B_{t/L})\geq\mu(H_{t/L})$$
for any half-plane with measure $\mu(B)=\mu(H)=\frac{1}{2}$. We calculate with the half-plane $\{y\mid y_{1}<0\}$. So
\begin{align*}
\mu(H_{t/L})&=\frac{N^{1/2}}{\sqrt{2\pi}}\int_{-\infty}^{\frac{t}{L}}e^{-\frac{Ns^{2}}{2}}ds\\
&=1-\frac{N^{1/2}}{\sqrt{2\pi}}\int_{\frac{t}{L}}^{\infty}e^{-\frac{Ns^{2}}{2}}ds\\
&=1-O\left(e^{-\frac{N t^{2}}{2L^{2}}}\right).\end{align*}
That is,
$$\mu\{c\in \R^{N}:G(c)<M+t\}=1-O\left(e^{-\frac{Nt^{2}}{2L^{2}}}\right)$$
so
$$\mu\{c\in\R^{N}:G(c)\geq M+t\}\leq Ce^{-\frac{Nt^{2}}{2L^{2}}}$$
that is \eqref{concen1}. Equation \eqref{concen2} follows from a similar argument. 
\end{proof}

We are now in a position to explore the relationship between $\phi^{-1}\left(\E[\phi(X)]\right)$ and $\E[X]$. We model our arguments on those of Burq-Lebeau \cite{BL13} (Theorems 4 and 5) when $\phi(\tau)=\tau^{p}$. We say that $\phi$ is admissible if:
\begin{enumerate}
\item $\phi\in C^{1}$ and strictly increasing.
\item $\phi(0)=0$.
\item $\phi(|\cdot|)$ obeys a generalised Minkowski inequality (with respect to $\mu$ the joint Gaussian measure)
\begin{equation}
\phi^{-1}\left(\int\phi(|X(c)+Y(c)|)d\mu\right)\leq \phi^{-1}\left(\int\phi(|X(c)|)d\mu\right)+\phi^{-1}\left(\int\phi(|Y(c)|)d\mu\right).\label{genMin}\end{equation}
\end{enumerate}
Notice that $\tau^{p}$ clearly satisfies all requirements. There are a number of different conditions that could be placed on $\phi$ to ensure (3). For example assuming (in addition to (1) and (2)) that both the functions $\phi(|\cdot|)$ and $-\log|\phi'(e^{\tau})|$ are convex is sufficient (see \cite{Mulholland50} for this and other possible conditions). 

\begin{thm}\label{thm:mcphi}
Suppose that $\phi$ is admissible, $|\phi'(\tau)|\leq C|\tau|^{q-1}$ for some $q\geq{}1$. Let $c=[c_{1},\dots,c_{N}]^{T}$ where each $c_{i}$ is a Gaussian random variable with $\sigma^{2}=N^{-1}$. Now if   $X=X(c)$ is a non-negative, random variable that obeys the Lipschitz bound
$$|X(c)-X(d)|\leq L\norm{c-d}_{\ell^{2}}$$
with $L\leq N^{1/2-\kappa}$ for some $\kappa>0$ then
$$\E\left[X\right]=\phi^{-1}\left(\E[\phi(X)]\right)+O\left(\phi^{-1}\left(\frac{C}{N^{\kappa q}}\Gamma(q/2)\right)\right).$$
In particular  if (for small enough $\epsilon$)
$$\phi^{-1}\left(\frac{C}{N^{\kappa q}}\Gamma(q/2)\right)\leq \epsilon \phi^{-1}\left(\E[\phi(X)]\right)$$ then there are $C_{1},C_{2}\in\R^{+}$ so that
\begin{equation}C_{1}\phi^{-1}\left(\E[\phi(X)]\right)\leq \E\left[X\right]\leq C_{2}\phi^{-1}\left(\E[\phi(X)]\right).\label{eqn:phimove}\end{equation}
\end{thm}

\begin{rmk}
The parameter $\kappa$ controls the Lipschitz bound. Later (see \ref{kappan} and \ref{kappanG}) for specific random variables we will determine suitable $\kappa$.
\end{rmk}

\begin{proof}
Let $M_{X}$ be a median for the random variable $X$. We will proceed by relating both $\phi^{-1}(\E[\phi(X)])$ and $\E[X]$ to $M_{X}$. Consider
\begin{equation}\phi\left(\left|\phi^{-1}(\E[\phi(X)])-M_{X}\right|\right)=\phi\left(\left|\phi^{-1}\left(\int\phi(X)d\mu\right)-\phi^{-1}\left(\int\phi(M_{X})d\mu\right)\right|\right).\label{eq:Ephi}\end{equation}
Note that by using the same argument that yields the reverse triangle inequality from the triangle inequality we have that
$$\left|\phi^{-1}\left(\int\phi(X)d\mu\right)-\phi^{-1}\left(\int\phi(M_{X})d\mu\right)\right|\leq\phi^{-1}\left(\int\phi(|X-M_{X}|)d\mu\right).$$
So
\begin{align*}
\phi\left(\left|\phi^{-1}(\E[\phi(X)])-M_{X}\right|\right)&\leq\int\phi(|X-M_{X}|)d\mu\\
&=\int_{0}^{\infty}\phi'(s)\mu(\{|X-M_{X}|\geq{}s\})ds\\
&\leq C\int_{0}^{\infty}s^{q-1}e^{-\frac{Ns^{2}}{2L^{2}}}ds.\end{align*}
If $L\leq N^{1/2-\kappa}$ we then get
$$\phi\left(\left|\phi^{-1}(\E[\phi(X)])-M_{X}\right|\right)\leq C\int_{0}^{\infty}s^{q-1}e^{-cN^{2\kappa}s^{2}}ds= \frac{C}{N^{\kappa q}}\Gamma(q/2).$$
On the other hand
\begin{align*}
|\E[X]-M_{X}|&\leq \int|X-M_{X}|d\mu\\
&=\int_{0}^{\infty}\mu(|X-M_{X}|\geq{}s)ds\\
&\leq \int_{0}^{\infty}e^{-cN^{2\kappa}s^{2}}ds\\
&\leq \frac{C}{N^{\kappa}}\int_{0}^{\infty}e^{-s^{2}}ds\leq \frac{C}{N^{\kappa}}.\end{align*}
Now since $\frac{C}{N^{\kappa}}\leq \frac{C}{N^{\kappa q}}\Gamma(q/2)$ we have that
$$|E[X]-M_{X}|\leq \frac{C}{N^{\kappa q}}\Gamma(q/2).$$
Therefore
$$E[X]=\phi^{-1}\left(E[\phi(X)]\right)+O\left(\phi^{-1}\left(\frac{C}{N^{\kappa q}}\Gamma(q/2)\right)\right).$$
\end{proof}

Since this results ensures that the expectation and median are very close together we can also state the concentration of measure results in terms of the expectation.

\begin{cor}\label{cor:expconcen}
Under the conditions of Theorem \ref{thm:mcphi},
\begin{equation}
\mu\left(\{c: |X(c)-\phi^{-1}\left(\E[\phi(X(c))]\right)|\geq{}t\}\right)\leq \exp(-CN^{2\kappa}t^{2})\label{eqn:expconcen}\end{equation}
for all $t\geq{}2\phi^{-1}\left(\frac{C}{N^{\kappa q}}\Gamma(q/2)\right)$.
\end{cor}

\begin{proof}
We saw in the proof of Theorem \ref{thm:mcphi} that
$$\left|\phi^{-1}\left(\E[\phi(X)]\right)-M_{X}\right|\leq \phi^{-1}\left(\frac{C}{N^{\kappa q}}\Gamma(q/2)\right).$$
So if 
$$|X-\phi^{-1}\left(\E[\phi(X)]\right)|\geq{}t$$
with $t\geq{}\phi^{-1}\left(\frac{C}{N^{\kappa q}}\Gamma(q/2)\right)$,
$$|X-M|\geq{}t/2.$$
The corollary follows the from concentration of measure for Gaussian variables. 

\end{proof}

\section{The X-ray transform of $|u|^{2}$}\label{sec:radon}
In this section we study the random variable obtained by taking a X-ray transform of $|u|^{2}$ along a line segment $\gamma$. That is we define the random variable
\begin{equation}F_{\beta}(x,\xi)=F_{\beta}(x,\xi;c)=\left(\int_{\gamma_{(x,\xi)}}|u|^{2}dl_{\gamma}\right)^{1/2}=\norm{u}_{L^{2}(\gamma_{(x,\xi)})}\label{eqn:Xray}\end{equation}
where $\gamma_{(x,\xi)}$ is the unit line segment, properly contained in $B_{1}(0)$, from $x$ in direction $\xi$. We will show that this random variable is equidistributed both in the sense that for fixed $(x,\xi)$, $F_{\beta}(x,\xi)$ is equidistributed but also in the stronger sense that the probability that $F_{\beta}(x,\xi)$ fails equidistribution for any $(x,\xi)$ is exponentially small. Since $\gamma_{(x,\xi)}$ is of unit length a uniform equidistribution statement for $F_{\beta}$ takes the following form. For every $(x,\xi)$, 
$$\E[F_{\beta}(x,\xi)]=1+o(1)$$
and
$$Pr\{c: \exists (x,\xi) \text{ so that } |F_{\beta}(x,\xi;c)-1|\geq{}m(h)\}\leq C\exp\left(cN^{2\kappa}m^{2}(h)\right) \kappa>0 $$
here $m(h)\to 0$ as $h\to 0$ but $m(h)>N^{-\kappa+\epsilon}$ to ensure exponential decay. In Theorem \ref{thm:Xrayunif} we will obtain explicit representations for valid $\kappa$, $m(h)$ and an explicit decay for the error term in the expectation.

\begin{thm}\label{thm:Xrayexp}
Suppose that for every $(x,\xi)\in B_{1}(0)\times \S^{n-1}$ the random variable $F_{\beta}(x,\xi)$ is given by \eqref{eqn:Xray}. Then  for $n=2$
\begin{equation}\E\left[F_{\beta}(x,\xi)\right]=1+O\left(\frac{1}{N^{\frac{1}{2}-\frac{\beta}{4}}}\right),\label{eqn:expecF2}\end{equation}
for $n=3$
\begin{equation}\E\left[F_{\beta}(x,\xi)\right]=1+O\left(\frac{\sqrt{\log(N)}}{N^{1/2}}\right)\label{eqn:expecF3}\end{equation}
and for  $n\geq{}3$,
\begin{equation}\E\left[F_{\beta}(x,\xi)\right]=1+O\left(\frac{1}{N^{1/2}}\right).\label{eqn:expecFn}\end{equation}
\end{thm}

\begin{proof}
We will use the results of Section \ref{sec:mc} to compute $\E\left[\norm{u}_{L^{2}(\gamma_{(x,\xi)})}\right]$ by computing the $\E\left[\norm{u}_{L^{2}(\gamma_{(x,\xi)})}^{2}\right]$. In this case we are using $\phi(\tau)=\tau^{2}$. Now
$$\E\left[\norm{u}_{L^{2}(\gamma_{(x,\xi)})}^{2}\right]=\E\left[\sum_{j,l}c_{j}c_{l}\int_{\gamma_{(x,\xi)}}e^{\frac{i}{h}\langle y,\xi_{j}-\xi_{l}\rangle}dy\right]$$
and since the Gaussian variables $c_{j}$ are mean $0$ and $\sigma^{2}=N^{-1}$,
$$\E\left[\norm{u}_{L^{2}(\gamma_{(x,\xi)})}^{2}\right]=1.$$
So we need only control the Lipschitz norm associated with $F_{\beta}(x,\xi)$. 
$$|F_{\beta}(x,\xi;c)-F_{\beta}(x,\xi;c)|\leq |\norm{u(c)}_{L^{2}(\gamma_{(x,\xi)})}-\norm{u_{d}}_{L^{2}(\gamma_{(x,\xi)})}|\leq\norm{u(c)-u(d)}_{L^{2}(\gamma_{(x,\xi)})}.$$
If we can estimate $\norm{u(c)-u(d)}_{L^{2}(\gamma_{(x,\xi)})}$ in terms of $\norm{u(c)-u(d)}_{L^{2}(B_{1}(0))}$ the almost orthogonality of the $e^{\frac{i}{h}\langle y,\xi_{j}\rangle}$ induced by the order $h$ spacing will yield a suitable Lipschitz constant. Since $u$ is a linear combination of plane waves with frequencies close to $\frac{1}{h}$ it makes sense to consider $u$ as an approximate solution (or quasimode) to
$$(h^{2}\Delta-1)u=0.$$
In particular $u$ is a quasimode of order $h^{\beta}$, that is
$$\norm{(h^{2}\Delta-1)u}_{L^{2}(B_{1}(0))}=O_{L^{2}}(h^{\beta}).$$
This characterisation allows us to use results on the $L^{p}$ growth of quasimodes on submanifolds \cite{Tacy10}. Note that, while we restrict our attention to $\1_{B_{1}(0)}u$ for exact computation, in cases (such as this) where we are only concerned with bounds we may instead work with $\chi(x)u(x)$ where $\chi$ is smooth, $\chi=1$ on $B_{1}(0)$ and is zero outside $B_{2}(0)$. The function $\chi(x)u(x)$ remains a quasimode of order $h^{\beta}$ for any $0\leq \beta\leq 1$. Similarly if we are only concerned with upper estimates we may extend the line segment $\gamma(x,\xi)$ so that the ends lie outside of the support of $\chi(x)u(x)$.

In the case where $\beta=1$ we can immediately use the submanifold estimates of \cite{Tacy10} (Theorem 1.7).  We can retrieve estimates for other values of $\beta$ by scaling. Let
$$u_{\beta}=u(h^{1-\beta}x)$$
and $Er[u]=(h^{2}\Delta-1)u$. Then
$$(h^{2\beta}\Delta-1)u_{\beta}=Er[u](h^{1-\beta}x).$$
Since $u$ is an $O_{L^{2}}(h^{\beta})$ quasimode of $(h^{2}\Delta-1)$ then $u_{\beta}$ is an $O_{L^{2}}(h^{\beta})$ quasimode of $(h^{2\beta}\Delta-1)$. So applying \cite{Tacy10} at the semiclassical scale $h^{\beta}$ we obtain
$$\norm{u_{\beta}}_{L^{2}(\gamma_{(x,\xi)})}\lesssim h^{-\beta\delta(n)}\norm{u_{\beta}}_{L^{2}}$$
\begin{equation}\delta(n)=\begin{cases}
\frac{1}{4}& n=2\\
\frac{n-1}{2}-\frac{1}{2}&n>3.\end{cases}\label{eq:deltan}\end{equation}
When $n=3$
$$\norm{u_{\beta}}_{L^{2}(\gamma_{(x,\xi)})}\lesssim h^{-\frac{\beta}{2}}\sqrt{\log(1/h)}\norm{u_{\beta}}_{L^2}$$
and so
\begin{equation}\norm{u}_{L^{2}(\gamma_{(x,\xi)})}\lesssim h^{-\beta\delta(n)+\frac{\beta-1}{2}}\norm{u}_{L^{2}}\label{eqn:gammaLip}\end{equation}
in dimensions other than three (it holds with a log loss in dimension three). When the $\xi_{j}$ actually lie on the sphere we can use \cite{CSogge14} to remove the log loss. 
Therefore the Lipschitz constant for $F_{\beta}(x,\xi)$ is $h^{-\beta\delta(n)+\frac{\beta-1}{2}}$ (with the log loss for $n=3$). Recall that $N$ grows as $h^{\beta-n}$. So  $L\leq N^{1/2-\kappa(n)}$,
\begin{equation}\kappa(n)=\begin{cases}
\frac{1}{2}-\frac{\beta}{4}&n=2\\
\frac{1}{2}&n>3\end{cases}\label{kappan}\end{equation}
and again with $\kappa(3)=1/2$ the same bounds hold with a log loss in dimension three. Therefore using Theorem \ref{thm:mcphi} we have that for fixed $(x,\xi)$ and $n=2$
$$\E\left[F_{\beta}(x,\xi)\right]=1+O\left(\frac{1}{N^{\frac{1}{2}-\frac{\beta}{4}}}\right),$$
for $n=3$
$$\E\left[F_{\beta}(x,\xi)\right]=1+O\left(\frac{\sqrt{\log N}}{N^{1/2}}\right),$$
for  $n\geq{}3$,
$$\E\left[F_{\beta}(x,\xi)\right]=1+O\left(\frac{1}{N^{1/2}}\right).$$
\end{proof}

\begin{thm}\label{thm:Xrayunif}
Let $F_{\beta}(x,\xi)$ be the random variable as given in Theorem \ref{thm:Xrayexp} and $\kappa(n)$ be given by \eqref{kappan}. For $m:\R^{+}\to\R^{+}$, $m(h)\to 0$ as $h\to 0$ and 
$$m(h)\geq{}N^{-\kappa(n)+\epsilon}=\begin{cases}
h^{\left(2-\beta\right)\left(\frac{1}{2}-\frac{\beta}{4}\right)-\epsilon(n-\beta)}&n=2\\
h^{\frac{n-\beta}{2}-\epsilon(n-\beta)}&n\geq{}3.\end{cases}$$
 define the  the exception set $S(m)$ by
$$S(m)=\{c: \exists (x,\xi) \text{ so that } |F_{\beta}(x,\xi;c)-1|\geq{}m(h)\}.$$
Let $\mu$ be the joint Gaussian measure given by \eqref{Gaussianmeasure}. Then the following estimate on the size of $S(m)$ holds.
\begin{equation}Pr(S(m))=\mu(S(m))\leq \exp\left(-N^{2\kappa(n)}m^{2}(h)\right).\label{eqn:exest}\end{equation}

\end{thm}

\begin{proof}
We use similar argument to that used in \cite{HanTac20} and \cite{dCI20} to obtain uniform equidistribution statements on small balls. In those works uniform equidistribution results for a random variable $X=X(y;c)$ depending on a parameter $y$ are obtained via a three step process.
\begin{enumerate}
\item A finite grid of $y^{\nu}$ is produced so that $X(y;c)$ is approximated sufficiently by $X(y^{\nu};c)$. This reduces the problem about the size of $S(m)$ to the size of a finite union of sets (where $S(m)$ is contained in the union).
\item The size of the exception sets are estimated at each grid-point. Usually this is achieved via a concentration of measure argument so that the decay is exponential in $h^{-1}$.
\item The number of grid-points is estimated. Usually this number will grow only polynomially in $h^{-1}$. Since the growth rate of the grid-points is overwhelmed by the exponential decay in each term a good union bound can be obtained. 
\end{enumerate}

In this case we are considering $|F_{\beta}(x,\xi;c)|^{2}$ as a random variable depending on the parameter $(x,\xi)$. Therefore we will form a grid of $(x^{\nu},\xi^{\nu})$ and approximate $|F_{\beta}(x,\xi;c)|^{2}$ by $|F_{\beta}(x^{\nu},\xi^{\nu};c)|^{2}$. We have then reduced the problem about the size of $S(m)$ to the size of a finite union
$$\bigcup_{\nu}\{c: |F_{\beta}(x^{\nu},\xi^{\nu};c)-1|\geq{}m(h)/2\}.$$
We find that the number of grid points $(x^{\nu},\xi^{\nu})$ necessary to obtain a suitable estimate does indeed grow polynomially. On the other hand concentration of measure arguments will allow us to prove exponential decay on the size of each of the $\{c: |F_{\beta}(x^{\nu},\xi^{\nu};c)-1|\geq{}m(h)/2\}$, overwhelming the growth from the number of gridpoints. Therefore the union bound decays exponentially.

Suppose that $c\in S(m)$, then there is some $(x,\xi)$ for which either
$$F^{2}_{\beta}(x,\xi;c)\geq{}(1+m(h))^{2}\quad\text{or}\quad F^{2}_{\beta}(x,\xi;c)\leq (1-m(h))^{2}.$$ 
We approximate $F^{2}_{\beta}(x,\xi;c)$ by $F^{2}_{\beta}(x^{\nu},\xi^{\nu};c)$ via a Taylor series. 
Writing $F^{2}_{\beta}(x,\xi;c)$ parametrically we see that
$$\left|\frac{\partial}{\partial x_{j}}F^{2}_{\beta}(x,\xi;c)\right|\leq 2\int_{0}^{1} |u_{x_{j}}(x+s\xi)||u(x+s\xi)|ds$$
$$\left|\frac{\partial}{\partial \xi_{j}}F^{2}_{\beta}(x,\xi;c)\right|\leq 2\int_{0}^{1} |u_{x_{j}}(x+s\xi)||u(x+s\xi)|ds.$$
So applying Cauchy Schwartz
$$\left|\frac{\partial}{\partial x_{j}}F^{2}_{\beta}(x,\xi;c)\right|+\left|\frac{\partial}{\partial \xi_{j}}F^{2}_{\beta}(x,\xi;c)\right|\leq C\norm{u_{x_{j}}}_{L^{2}(\gamma(x,\xi))
}\norm{u}_{L^{2}(\gamma(x,\xi))}.$$
Since $u_{x_{j}}$ is also a $O(h^{\beta})$ quasimode of $(h^{2}\Delta-1)$ we have that
$$F^{2}_{\beta}(x,\xi;c)=F^{2}_{\beta}(x^{\nu},\xi^{\nu};c)+O\left(h^{-2\beta\delta(n)}|(x,\xi)-(x^{\nu},\xi^{\nu})|\right)$$
where $\delta(n)$ is given by \eqref{eq:deltan}. Therefore in dimension $n=2$ if we place our grid points with spacing $h^{\frac{\beta}{2}+\left(2-\beta\right)\left(\frac{1}{2}-\frac{\beta}{4}\right)}$ we can always write (for some $(x^{\nu},\xi^{\nu})$)
$$F^{2}_{\beta}(x,\xi;c)=F^{2}_{\beta}(x^{\nu},\xi^{\nu};c)+O(N^{-\kappa(n)})=F^{2}_{\beta}(x^{\nu},\xi^{\nu};c)+O\left(h^{\left(n-\beta\right)\left(\frac{1}{2}-\frac{\beta}{4}\right)}\right).$$
Similarly for $n\geq{}3$ we set the spacing at $h^{n-2+\frac{n-\beta}{2}}$ to have
$$F^{2}_{\beta}(x,\xi;c)=F^{2}_{\beta}(x^{\nu},\xi^{\nu};c)+O(N^{-\kappa(n)})=F^{2}_{\beta}(x^{\nu},\xi^{\nu};c)+O\left(h^{\frac{n-\beta}{2}}\right).$$
Recall that $m(h)\geq{}N^{-\kappa(n)+\epsilon}=h^{(n-\beta)(\kappa(n)-\epsilon)}$. So (at least for small enough $h$) in both cases there is a grid point $(x^{\nu},\xi^{\nu})$ where
$$|F_{\beta}(x^{\nu},\xi^{\nu};c)-1|\geq{}m(h)/2.$$
For $\phi(\tau)=\tau^{2}$, we already saw that $\phi^{-1}\left(E[\phi(F(x^{\nu},\xi^{\nu}))]\right)=1$ so we can now apply Corollary \ref{cor:expconcen} to obtain
\begin{equation}\mu\left(\{c: |F_{\beta}(x^{\nu},\xi^{\nu};c)-1|\geq{}m(h)/2\}\right)\leq \exp(-CN^{2\kappa}m^{2}(h)).\label{excep}\end{equation}
\end{proof}

Putting together \eqref{excep} with the results of Theorem \ref{thm:Xrayexp},  the probability that there there is any line $\gamma(x,\xi)$ over which equidistribution fails is so small that we would not expect to be able to see strong scars in numerical simulations. In the next section we consider the question of phase-space equidistribution. We obtain two distinct results. The first says that when we are far away from the Planck-scale equidistribution holds in the uniform sense. The more delicate result tells us that at the Planck scale phase-space equidistribution fails with high probability.

\section{Phase-space concentrations}\label{sec:phasespace}

We now turn our attention to the question of equidistribtion in phase space. That is if $P$ is a semiclassical localiser, is $\norm{Pu}_{L^{2}}$ equidistributed? We say that $P$ is a semiclassical localiser if 
$$P=Op(p(y,\eta;h))=p_{h}(y,hD)=\frac{1}{(2\pi h)^{n}}\int e^{\frac{i}{h}\langle y-z,\eta\rangle}p(y,\eta;h)u(z)dzd\eta$$
where $p(y,\eta;h)$ is compactly supported in $\R^{n}\times \R^{n}$. The support of $p(y,\eta;h)$ may depend on $h$ but is constrained by the uncertainty principle,
$$|\nabla_{x} p\cdot\nabla_{\eta}p|\leq h^{-1}.$$
For the purposes of this paper we will be assuming that $p(y,\eta;h)$ is supported near a point $(x,\xi)$. We will introduce two parameters $\alpha$ and $\mu$ to measure how near.
\begin{enumerate}
\item The $h^{\alpha}$ scale controls the level of angular frequency localisation. The radial frequency is localised so that $|\eta|\leq 4$. We can ignore contributions from frequencies $|\eta|\geq{}4$ since $u$ is a sum of plane waves with frequencies in $[1-h^{\beta},1+h^{\beta}]$, $\beta\geq{}0$. 
\item The parameter $\mu$ controls how close/far we are from Planck  scale. When $\mu=1$ we are exactly at Planck scale and therefore at the limits of the uncertainty principle. Large $\mu$ (growing in $h^{-1}$) correspond to more relaxed localisation. 
\end{enumerate}

Let $(x,\xi)\in \R^{n}\times \S^{n-1}$.  Because of the curvature we obtain linearisation when $\alpha=\beta/2$ so we are interested in the cases $\alpha\leq \beta/2$. Figure \ref{fig:supp} shows the support of $p^{\alpha,\mu}_{(x,\xi)}(y,\eta)$ in both configuration and Fourier space. 

\begin{figure}[h!]
\includegraphics[scale=1]{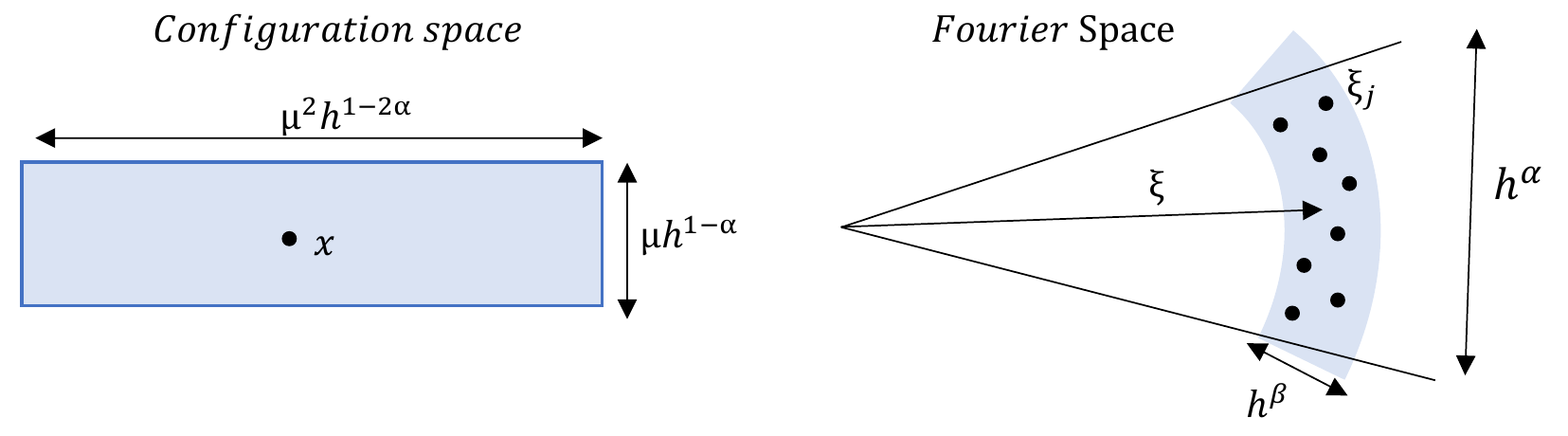}
\caption{The configuration an Fourier supports of the symbol $p^{\alpha,\mu}_{(x,\xi)}(y,\eta)$.}\label{fig:supp}
\end{figure}
Define
$$P_{(x,\xi)}^{\alpha,\mu}=p_{(x,\xi)}^{\alpha,\mu}(y,hD)=Op(p_{(x,\xi)}^{\alpha,\mu}(y,\eta))$$
where
\begin{multline}
p_{(x,\xi)}^{\alpha,\mu}(y,\eta)=Ah^{-\frac{n}{2}+\alpha}\mu^{-\frac{n+1}{2}}\chi\left(\mu^{-2}h^{-1+2\alpha}|\langle x-y,\xi\rangle|\right)\\
\times\chi\left(\mu^{-1}h^{-1+\alpha}|(x-y)-\langle x-y,\xi\rangle\xi|\right)\chi\left(\frac{|\eta|}{4}\right)\chi\left(h^{-\alpha}\left|\frac{\eta}{|\eta|}-\xi\right|\right).\label{psymdef}\end{multline}
The prefactor is chosen so that
$$A^{2}h^{1-\beta+(n-1)(1-\alpha)}\sum_{j}\chi^{2}\left(h^{-\alpha}\left|\xi_{j}-\xi\right|\right)\int \chi^{2}\left(|\langle y,\xi\rangle|\right)\chi^{2}\left(|y-\langle y,\xi\rangle|\right)dy=1.$$
The angular cut off will ensure that $A$ is of order one. Note that
$$\norm{P^{\alpha,\mu}_{(x,\xi)}u(\cdot)}_{L^{2}}^{2}=\langle P^{\alpha,\mu}_{(x,\xi)}u(\cdot),P^{\alpha,\mu}_{(x,\xi)}u(\cdot)\rangle=\langle u(\cdot),(P^{\alpha,\mu}_{(x,\xi)})^{\star}P^{\alpha,\mu}_{(x,\xi)}u(\cdot)\rangle.$$
So if the random wave is assumed to equidistribute down to small scale (in terms of semiclassical defect measures) then this normalisation ensures that 
$$\norm{P^{\alpha,\mu}_{(x,\xi)}u(\cdot)}_{L^{2}}^{2}= 1+O(\mu^{-1}).$$
 Consider the random variable $G_{\alpha,\beta,\mu}(x,\xi)$ given by
\begin{equation}
G_{\alpha,\beta,\mu}=\norm{P^{\alpha,\mu}_{(x,\xi)}u(\cdot)}_{L^{2}}.\label{eqn:Gdef}\end{equation}
If we can show that 
$$\E[\norm{G_{\alpha,\beta,\mu}(x,\xi)}_{L^{\infty}(B_{1}(0)\times \S^{n-1})}]\geq{}c\sqrt{\log(1/h)}$$
we establish that, apart from a set whose measure decays to zero in $h$, equidistribution fails at small scales. We will show this is true for small $\mu$, that is $\mu\leq C$. We will then see that for large $\mu$, $\mu\geq{}h^{-\epsilon}$ that in fact uniform equidistribution results hold. So it is only at Planck scale that we will see fluctuations.

As in Burq-Lebeau \cite{BL13} we will first control $\norm{G_{\alpha,\beta,\mu}}_{L^{\infty}}$ by $\norm{G_{\alpha,\beta,\mu}}_{L^{p_{h}}}$ for $p_{h}=\delta\log(h)$. To compute $\norm{G_{\alpha,\beta,\mu}}_{L^{p_{h}}}$ we again use Theorem \ref{thm:mcphi}. In this case we set $\phi(\tau)=\tau^{p_{h}}$ and compute $\norm{G_{\alpha,\beta,\mu}}_{L^{p_{h}}}^{p_{h}}$. To facilitate that calculation we will assume that $p_{h}$ is an even integer (to allow us to expand the expression for $G_{\alpha,\beta,\mu}$).

\begin{prop}\label{prop:infcont}
Suppose $G_{\alpha,\beta,\mu}(x,\xi)$ is given by \eqref{eqn:Gdef}. Then for $p_{h}=\delta\log(1/h)$ (for some $\delta>0$), there are constants $C_{1},C_{2}$ so that,
\begin{equation}C_{1}\norm{G_{\alpha,\beta,\mu}}_{L^{p_{h}}(B_{1}(0)\times S^{n-1})}\leq \norm{G_{\alpha,\beta,\mu}}_{L^{\infty}(B_{1}(0)\times S^{n-1})}\leq C_{2} \norm{G_{\alpha,\beta,\mu}}_{L^{p_{h}}(B_{1}(0)\times S^{n-1})}.\label{infinitycontrol}\end{equation}
\end{prop}

\begin{proof}
Since $B_{1}(0)\times \S^{n-1}$ is compact, the first half of \eqref{infinitycontrol} follows from H\"{o}lder. To get the second half we will in fact control the $L^{\infty}$ norm of $G^{2}_{\alpha,\beta,\mu}$ in terms of its $L^{p_{h}/2}$ norm. The requisite control comes directly from the regularity properties of the symbol $p^{\alpha,\mu}_{(x,\xi)}(y,\eta)$. The semiclassical Sobolev estimates tell   us that if $W:\R^{d}\to\R$ is semiclassically localised,
$$W=\chi(w,hD)W+O_{L^{2}}(h)$$
for some $\chi(w,\omega)$ with compact support then
$$\norm{W}_{L^{\infty}}\leq Ch^{-\frac{d}{p}}\norm{W}_{L^{2}}.$$
Saying that $W$ is semiclassically localised is the same as saying that $W$ is localised in space and that (up to Schwartz type error) its Fourier transform is localised in the ball of radius $1/h$. Notice that 
$$|D^{\gamma_{1}}_{x}D^{\gamma_{2}}_{\xi}p^{\alpha,\mu}_{(x,\xi)}|\leq Ch^{-(|\gamma_{1}|+|\gamma_{2}|)}.$$
Therefore indeed outside the ball of radius $1/h$ the Fourier transform of $p^{\alpha,\mu}_{(x,\xi)}$ decays in a Schwartz fashion and we can apply the semiclassical Sobolev estimates. In this case $d=2n-1$ and if we set $p_{h}=\delta\log(1/h)$ we obtain
$$\norm{G^{2}_{\alpha,\beta,\mu}}_{L^{\infty}}\leq Ce^{\frac{4n-2}{c}}\norm{G^{2}_{\alpha,\beta,\mu}}_{L^{p_{h}/2}}.$$

\end{proof}

\begin{thm}\label{thm:concenphase}
Let $G_{\alpha,\beta,\mu}(x,\xi)$ be given by \eqref{eqn:Gdef}. For $\mu\leq C$ and small enough (but independent of $h$)  there exist constants $C_{1},C_{2}$ so that
\begin{equation}C_{1}\sqrt{\log(1/h)}\leq \E\left[\norm{G_{\alpha,\beta,\mu}}_{L^{\infty}(B_{1}(0)\times \S^{n-1})}\right]\leq C_{2}\sqrt{\log(1/h)}.\label{eqn:Pexpectinfty}\end{equation}

\end{thm}

\begin{proof}
From the results of Proposition \ref{prop:infcont} it is enough to find upper and lower bounds for $\E\left[\norm{G_{\alpha,\beta,\mu}}_{L^{p_{h}}}\right]$ for $p_{h}\approx\delta\log(1/h)$. Therefore let  $p_{h}$ be an even integer obeying $\frac{\delta}{2}\log(1/h)\leq p_{h}\leq \delta\log(1/h)$.
Let $\phi(\tau)=\tau^{p_{h}}$, we will compute
$$\phi^{-1}\left(\E\left[\phi\left(\norm{G_{\alpha,\beta,\mu}}_{L^{p_{h}}(B_{1}(0)\times \S^{n-1})}\right)\right]\right)$$
and apply the measure concentration results of Theorem \ref{thm:mcphi}.

Since we have assumed $p_{h}$ is even we can write
$$|G_{\alpha,\beta,\mu}(x,\xi)|^{p_{h}}=\left(|G_{\alpha,\beta,\mu}(x,\xi)|^{2}\right)^{p_{h}/2}.$$
Since
$$u=\sum_{\lambda_{j}\in\Lambda_{\beta}}c_{j}e^{\frac{i}{h}\langle y,\xi_{j}\rangle}$$
$G^{2}_{\alpha,\beta,\mu}(x,\xi)$ is quadratic in the column vector $c=[c_{1},\dots,c_{N}]^{T}$. We will,  similar to \cite{dCI20}, scale and perform transformations so that $G^{2}_{\alpha,\beta,\mu}(x,\xi)$ can be written as $y^{T}Dy$ where $D$ is a diagonal matrix and $y^{T}$ is a standard $N$-dimensional Gaussian variable (with variance one). 

Let $w=\sigma c=N^{-1/2}c$ we can write
$|G_{\alpha,\beta,\mu}(x,\xi)|^{2}$ as 
$$|G_{\alpha,\beta,\mu}(x,\xi)|^{2}=w^{T}\left[A_{\alpha,\beta,\mu}\right]\,w$$
where $A_{\alpha,\beta,\mu}$ is the matrix with entries 
\begin{equation}(A_{\alpha,\beta,\mu})_{j,k}=N^{-1}\int_{C_{1/2}}\psi_{j}(y)\overline{\psi}_{k}(y)dy\label{matA}\end{equation}
with
$$\psi_{j}(y)=p^{\alpha,\mu}(y,hD)e^{\frac{i}{h}\langle y, \xi_{j}\rangle }.$$
If we diagonalise $A_{\alpha,\beta,\mu}$ to write $A_{\alpha,\beta,\mu}=U^{T}DU$ where $U$ is unitary and $D$ is diagonal with entries $\lambda_{j}$  we can write
$$|G_{\alpha,\beta,\mu}(x,\xi)|^{2}=y^{T}Dy=\sum_{j}\lambda_{j}y_{j}^{2}$$
where  $y=Uw$. Since $U$ is unitary $y$ is also an $N$-dimensional Gaussian random variable with variance one. 
 Recall that
$$|G_{\alpha,\beta,\mu}(x,\xi)|^{2}=\norm{p^{\alpha,\mu}_{(x,\xi)}(y,hD)u}_{L^{2}}^{2}$$
so $A_{\alpha,\beta,\mu}$ must have only non-negative eigenvalues. Let $M$ be a large integer (eventually we will want $M\sim \log(1/h)$ so that $p_{h}=2M$). Then
\begin{align*}
|G_{\alpha,\beta,\mu}(x,\xi)|^{2M}&=\left(\sum_{j}\lambda_{j}y_{j}^{2}\right)^{M}\\
&=\sum_{[j_{1},\dots,j_{M}]}\prod_{k=1}^{M}\lambda_{j_{k}}y_{j_{k}}^{2}.\end{align*}
Let $J=[j_{1},\dots,j_{M}]$ and let $|J|_{d}$ be the number of distinct $j_{i}$ in $J$.
Then
$$|G_{\alpha,\beta,\mu}(x,\xi)|^{M}=\sum_{r=0}^{R}\sum_{2^{r-1}\leq|J|_{d}\leq 2^{r}}\prod_{k=1}^{M}\lambda_{j_{k}}y_{j_{k}}^{2}$$
with $2^{R}=M$. Since the $y_{j_{k}}^{2}$ are independent Gaussian variables we can compute
$$\E\left[\prod_{k=1}^{M}y_{j_{k}}^{2}\right]=\left(\frac{2}{\sqrt{\pi}}\right)^{M}\Gamma(p_{1}+1/2)\Gamma(p_{2}+1/2)\cdots \Gamma(p_{|J|_{d}}+1/2)$$
where the $p_{i}$ are the number of distinct occurrences in $J$. Obviously $p_{1}+p_{2}+\cdots+p_{|J|_{d}}=M$. So
$$\E\left[|G_{\alpha,\beta,\mu}(x,\xi)|^{M}\right]=\sum_{r=0}^{R}\sum_{2^{r-1}<|J|_{d}\leq{}2^{r}}\prod_{k=1}^{|J|_{d}}\lambda_{j_{k}}^{p_{k}}\Gamma(p_{k}+1/2).$$
We obtain uniform upper and lower bounds for $\E\left[|G_{\alpha,\beta,\mu}(x,\xi)|^{M}\right]$ which, alongside the compact nature of $B_{1}(0)\times \S^{n-1}$ carry over to give upper and lower bounds on $\E\left[\norm{G_{\alpha,\beta,\mu}(x,\xi)}_{L^{p_{h}}}^{p_{h}}\right]$. 
Since all the eigenvalues $\lambda_{j}$ are non-negative we can obtain lower bounds by obtain a lower bound for any fixed $r$. We claim that if $\mu\leq C$ (the Planck scale case) then there is some $\j$ so that
\begin{equation}\lambda_{\j}>c\text{Tr}(A_{\alpha,\beta,\mu}).\label{Tracelower}\end{equation}
In which case the term associated with $J=[\j,\j,\dots,\j]$ has the lower bound
$$\prod_{k=1}^{|J|_{d}}\lambda_{j_{k}}^{p_{k}}\Gamma(p_{k}+1/2)>c^{N}(\text{Tr}(A_{\alpha,\beta}))^{M}\Gamma(M+1/2)>C^{M}(\text{Tr}(A_{\alpha,\beta}))^{M}M^{M}.$$
To justify this claim we (in Lemma \ref{lem:tracebnds}) compute and compare $\Trace(A_{\alpha,\beta,\mu})$ and $\Trace(A^2_{\alpha,\beta,\mu})$. For $\mu\leq C$ we find that,
$$\frac{\left(\sum_{i}\lambda_{i}\right)^{2}}{\sum_{i}\lambda_{i}^{2}}\leq{}C$$
and therefore we can conclude \eqref{Tracelower}. On the other hand
$$\left(\frac{2}{\sqrt{\pi}}\right)^{M}\Gamma(p_{1}+1/2)\Gamma(p_{2}+1/2)\cdots \Gamma(p_{|J|_{d}}+1/2)\leq C^{M}M^{M}$$
for any combination of $p_{i}$ so
$$\E\left[|G_{\alpha,\beta,\mu}(x,\xi)|^{M}\right]\leq C^{M}M^{M}\left(\sum_{j}\lambda_{j}\right)^{M}=C^{M}M^{M}(\text{Tr}(A_{\alpha,\beta,\mu}))^{M}.$$
In Lemma \ref{lem:tracebnds} we see that there are constants $a_{1}$ and $a_{2}$ such that
$$a_{1}\leq \text{Tr}(A_{\alpha,\beta,\mu})\leq a_{2}$$
so if $p_{h}=2M$ we have that
$$C_{1}p_{h}^{1/2}\leq \left(\E\left[\norm{G_{\alpha,\beta,\mu}(x,\xi)}_{L^{p_{h}}(B_{1}(0)\times \S^{n-1})}^{p_{h}}\right]\right)^{\frac{1}{p_{h}}}\leq C_{2}p_{h}^{1/2}.$$

Now all we need is the concentration of measure that allows us to relate $\phi^{-1}\left(\E\left[\phi(\norm{G_{\alpha,\beta,\mu}}_{L^{p_{h}}})\right]\right)$ to $\E\left[\norm{G_{\alpha,\beta,\mu}}_{L^{p_{h}}}\right]$. In Lemma \ref{lem:Lipbnds} we find that for $\mu\leq C$ and 
$$G_{p}=\norm{G_{\alpha,\beta,\mu}}_{L^{p}(B_{1}(0)\times \S^{n-1})}$$
$$|G_{p}(c)-G_{p}(d)|\leq N^{\frac{1}{2}-\frac{1}{p}}\norm{c-d}_{\ell^{2}}.$$
So applying the results of Theorem \ref{thm:mcphi}
$$\E\left[\norm{G_{\alpha,\beta,\mu}}_{L^{p_{h}}(B_{1}(0)\times \S^{n-1})}\right]=\left(\E\left[\norm{G_{\alpha,\beta,\mu}}_{L^{p_{h}}(B_{1}(0)\times \S^{n-1})}^{p_{h}}\right]\right)^{1/p_{h}}+O\left(\frac{p_{h}^{\frac{1}{2}}}{N^{\frac{1}{p_{h}}}}\right).$$
Since $\frac{\delta}{2}\log(1/h)\leq p_{h}\leq \delta\log(1/h)$, 
$$\E\left[\norm{G_{\alpha,\beta,\mu}}_{L^{p_{h}}(B_{1}(0)\times \S^{n-1})}\right]=\left(\E\left[\norm{G_{\alpha,\beta,\mu}}_{L^{p_{h}}(B_{1}(0)\times \S^{n-1})}^{p_{h}}\right]\right)^{1/p_{h}}+O\left(p_{h}^{\frac{1}{2}}e^{-\frac{n-\beta}{\delta}}\right).$$
Therefore by making $\delta$ small enough we can obtain both upper and lower bounds, that is
$$C_{1}\sqrt{\log(1/h)}\leq \E\left[\norm{G_{\alpha,\beta,\mu}}_{L^{\infty}(B_{1}(0)\times \S^{n-1})}\right]\leq C_{2}\sqrt{\log(1/h)}.$$
\end{proof}

We now consider the case where we are far from Planck scale. That is where $\mu>h^{-\epsilon}$. In this case we find that the uniform version of equidistribution holds. 

\begin{thm}\label{thm:equiphase}
Let $G_{\alpha,\beta,\mu}(x,\xi)$ be given by \eqref{eqn:Gdef} and suppose $\mu\geq h^{-\epsilon}$, then
\begin{equation}
\E\left[G_{\alpha,\beta,\mu}(x,\xi)\right]=1+O\left(h^{\frac{\epsilon}{2}}\right).\label{eqn:expectinfty}\end{equation}
Further if $\delta>0$ and $m(h)\geq{}\max\left(h^{\epsilon-\delta}, h^{\frac{2\epsilon(n+1)}{4(n-\beta)}-\delta}\right)$,
\begin{equation}Pr\{c:\exists(x,\xi),\text{ so that } \left|G_{\alpha,\beta,\mu}(x,\xi)-1\right|>m(h)\}\leq \exp(-N^{\frac{2\epsilon(n+1)}{4(n-\beta)}}m^{2}(h)).\label{eqn:Guniform}\end{equation}

\end{thm}
\begin{proof}
We will follow much the same process as we did in Theorems \ref{thm:Xrayexp} and \ref{thm:Xrayunif} where we controlled the behaviour of $F(x,\xi)$. That is we will set $\phi(\tau)=\tau^{2}$ and use Theorem \ref{thm:mcphi} to compute $E\left[G_{\alpha,\beta,\mu}(x,\xi)\right]$ in terms of $\E\left[G^{2}_{\alpha,\beta,\mu}(x,\xi)\right]$. Note that
$$G^{2}_{\alpha,\beta,\mu}(x,\xi)=\sum_{j,l}c_{j}c_{l}\langle P^{\alpha,\mu}_{(x,\xi)}e^{\frac{i}{h}\langle y,\xi_{j}\rangle}, P^{\alpha,\mu}_{(x,\xi)}e^{\frac{i}{h}\langle y,\xi_{l}\rangle}\rangle.$$
Since the Gaussian variables have mean zero all non-diagonal terms fall out when we compute expectation so we only need compute
$$\sum_{j}\langle P^{\alpha,\mu}_{(x,\xi)}e^{\frac{i}{h}\langle y,\xi_{j}\rangle}, P^{\alpha,\mu}_{(x,\xi)}e^{\frac{i}{h}\langle y,\xi_{j}\rangle}\rangle.$$
Since $\mu\geq{}h^{-\epsilon}$ we can write
$$P^{\alpha,\mu}_{(x,\xi)}e^{\frac{i}{h}\langle y,\xi_{j}\rangle}=p^{\alpha,\mu}_{(x,\xi)}(y,\xi_{j})e^{\frac{i}{h}\langle y,\xi_{j}\rangle}+h^{\epsilon}r(y,\xi_{j})e^{\frac{i}{h}\langle y,\xi\rangle}$$
where $r(y,\xi_{j})$ has the same support an normalisation properties as $p^{\alpha,\mu}_{(x,\xi)}$. The contribution from the top terms is
\begin{multline*}h^{2\alpha-\beta}\mu^{-(n+1)}A^{2}\sum_{j}\chi^{2}\left(h^{-\alpha}\left|\frac{\xi_{j}}{|\xi_{j}|}-\xi\right|\right)\int \chi^{2}\left(\mu^{-2}h^{-1+2\alpha}|\langle x-y,\xi\rangle|\right)\\
\times\chi^{2}\left(\mu^{-1}h^{-1+\alpha}|(x-y)-\langle x-y,\xi\rangle|\right)dy.\end{multline*}
Rescaling and translating/rotating so that $x=0$ and $\xi=e_{1}$ we see that the normalisation on $A$ ensures that this term is one. Since $r(y,\xi_{j})$ has the same normalisation properties as $p^{\alpha,\mu}(y,\xi_{j})$ we can conclude that
$$\sum_{j}\langle P^{\alpha,\mu}_{(x,\xi)}e^{\frac{i}{h}\langle y,\xi_{j}\rangle}, P^{\alpha,\mu}_{(x,\xi)}e^{\frac{i}{h}\langle y,\xi_{j}\rangle}\rangle=1+O(h^{\epsilon}).$$
Therefore 
$$\E\left[G^{2}_{\alpha,\beta,\mu}(x,\xi)\right]=1+O(h^{\epsilon}).$$
Now we can use the measure concentration to find $\E\left[G_{\alpha,\beta,\mu}(x,\xi)\right]$. In Lemma \ref{lem:Lipbnds} we find that if
$$G_{\infty}=\norm{G_{\alpha,\beta,\mu}(\cdot,\cdot)}_{L^{\infty}(B_{1}(0)\times \S^{n-1})}$$
then
$$|G_{\infty}(c)-G_{\infty}(d)|\leq N^{1/2}\mu^{-\left(\frac{n+1}{4}\right)}\leq N^{1/2}h^{\epsilon\left(\frac{n+1}{4}\right)}=N^{\frac{1}{2}-\kappa}$$
for \begin{equation}\kappa = \frac{\epsilon}{n-\beta}\left(\frac{n+1}{4}\right).\label{kappanG}\end{equation} So by Theorem \ref{thm:mcphi}
$$\E\left[G_{\alpha,\beta,\mu}(x,\xi)\right]=1+O(h^{\epsilon}+N^{-\kappa})=1+O\left(h^{\epsilon}+h^{\frac{2\epsilon(n+1)}{4(n-\beta)}}\right).$$
Note that the cut offs we used to define $p^{\alpha,\mu}_{(x,\xi)}$ have power (in $h^{-1}$) type regularity. So by the same arguments that we used in the proof of uniform equidistribution of $F(x,\xi)$ we can find a polynomial grid $(x^{\nu},\xi^{\nu})$ so that failure to equidistribute at some point $(x,\xi)\in \R^{n}\times \S^{n-1}$ implies a failure to equidistribute at a grid point $(x^{\nu},\xi^{\nu})$.   Then applying the concentration of measure at these points we obtain
$$Pr\{c: |G_{\alpha,\beta,\mu}(x^{\nu},\xi^{\nu})-1|\geq{}m(h)\}\leq \exp(-N^{2\kappa}m^{2}(h)).$$
Since the number of grid points only growth polynomially in $h^{-1}$ (compared to the exponential decay of measure) we obtain \eqref{eqn:Guniform}.

\end{proof}

\section{Technical Lemmata}

This section is devoted to the proofs of the two technical Lemmata used a number of times throughout the paper.

The first lemma controls the growth of $\Trace(A_{\alpha,\beta,\mu})$ and $\Trace(A_{\alpha,\beta,\mu}^{2})$ dependent on $\mu$ (the parameter that controls how far we are from the Planck scale).

\begin{lemma}\label{lem:tracebnds}
Suppose $A_{\alpha,\beta,\mu}$ is given by \eqref{matA}. Then there exist constants $a_{1},a_{2},a_{3},a_{4}>0$ so that

\begin{equation} a_{1}\leq \Trace(A_{\alpha,\beta,\mu})\leq a_{2}\label{Tracebnds}\end{equation}
and
\begin{equation}a_{3}\mu^{-n-1}\leq{}\Trace(A_{\alpha,\beta,\mu}^{2})\leq a_{4}\mu^{-n-1}.\label{Tracesqbnd}\end{equation}
\end{lemma}

\begin{proof}

From \eqref{matA} we have that,
$$\Trace(A_{\alpha,\beta,\mu})=h^{n-\beta}\sum_{j}\int |\psi_{j}(y)|^{2}dy$$
and
\begin{align*}\Trace(A_{\alpha,\beta,\mu}^{2})&=h^{2(n-\beta)}\sum_{j,m}\iint \psi_{j}(y)\overline{\psi}_{j}(y')\overline{\psi_{m}}(y)\psi_{m}(y')dydy'\\
&=h^{2(n-\beta)}\sum_{j,m}|I_{j,m}|^{2}\end{align*}
where
$$I_{j,m}=\int\psi_{j}(y)\overline\psi_{m}(y)dy.$$
Recall that
\begin{align*}
\psi_{j}(y)&=p^{\alpha,\mu}_{(x,\xi)}(y,hD)e^{\frac{i}{h} \langle y, \zeta_{j}\rangle}\\
&=h^{-\frac{n}{2}+\alpha}\mu^{-\frac{n+1}{2}}w^{\alpha,\mu}_{(x,\xi)}(y)q^{\alpha}_{\xi}(hD)e^{\frac{i}{h} \langle y, \zeta_{j}\rangle}\end{align*}
where
$$w^{\alpha,\mu}_{(x,\xi)}(y)=\chi\left(\mu^{-2}h^{-1+2\alpha}|\langle x-y,\xi\rangle|\right)\chi\left(\mu^{-1}h^{1-\alpha}|(x-y)-\langle x-y,\xi\rangle\xi|\right)$$
and
$$q^{\alpha}_{\xi}(\eta)=\chi\left(h^{-\alpha}\left|\frac{\eta}{|\eta|}-\xi\right|\right)\chi\left(\frac{|\eta|}{4}\right).$$
Therefore we can write
\begin{align*}
\psi_{j}(y)&=h^{-\frac{n}{2}+\alpha}\mu^{-\frac{n+1}{2}}w^{\alpha,\mu}_{(x,\xi)}(y)\FT_{h}^{-1}\left(q^{\alpha}_{\xi}(\cdot)\FT_{h}[e^{\frac{i}{h}\langle \cdot,\zeta_{j}\rangle}]\right)\\
&=h^{-\frac{n}{2}+\alpha}\mu^{-\frac{n+1}{2}}e^{\frac{i}{h}\langle y,\xi_{j}\rangle}w^{\alpha,\mu}_{(x,\xi)}(y)\chi\left(h^{-\alpha}\left|\frac{\xi_{j}}{|\xi_{j}|}-\xi\right|\right).\end{align*}
So we compute
\begin{align*}
I_{j,m}(x,\xi)&=\int \psi_{j}(y)\psi_{m}(y)dy\\
&=h^{-n+2\alpha}\mu^{-n-1}\chi\left(h^{-\alpha}\left|\frac{\xi_{j}}{|\xi_{j}|}-\xi\right|\right)\chi\left(h^{-\alpha}\left|\frac{\xi_{m}}{|\xi_{m}|}-\xi\right|\right)\int e^{\frac{i}{h}\langle y,\xi_{j}-\xi_{m}\rangle}(w^{\alpha,\mu}_{(x,\xi)}(y))^{2}dy.\end{align*}
First note that $I_{j,m}(x,\xi)=0$ if either 
$$\left|\frac{\xi_{j}}{|\xi_{j}|}-\xi\right|\geq{}2h^{\alpha}\quad\text{or}\quad\left|\frac{\xi_{m}}{|\xi_{m}|}-\xi\right|>2h^{\alpha}.$$
Immediately this tells us that there can only be $Ch^{\alpha(n-1)+\beta-n}$ of each where $I_{j,m}$ is nonzero. Conversely there is a smaller constant $a$ so that there are more than $ah^{\alpha(n-1)+\beta-n}$ points $\xi_{j}$ for which 
$$\chi\left(h^{-\alpha}\left|\frac{\xi_{j}}{|\xi_{j}|}-\xi\right|\right)=1.$$
If $\xi_{j}=\xi_{k}$ we have
$$I_{j,j}(x,\xi)=h^{-n+2\alpha}\mu^{-n-1}\chi^{2}\left(h^{-\alpha}\left|\frac{\xi_{j}}{|\xi_{j}|}-\xi\right|\right)\int(w^{\alpha,\mu}_{x,\xi}(y))^{2}dy.$$
Therefore there are constants $a_{1}$ and $a_{2}$ so that
$$a_{1}h^{-\alpha(n-1)}\chi^{2}\left(h^{-\alpha}\left|\frac{\xi_{j}}{|\xi_{j}|}-\xi\right|\right)\leq I_{j,j}(x,\xi)\leq a_{2}h^{-\alpha(n-1)}\chi^{2}\left(h^{-\alpha}\left|\frac{\xi_{j}}{|\xi_{j}|}-\xi\right|\right).$$
Since there are of order $h^{\beta-n}h^{\alpha(n-1)}$ $j$  that appear  in the sum we obtain constants (and here we abuse notation by allowing constants to vary line by line)
$$a_{1}\leq{}\Trace(A_{\alpha,\beta})\leq{}a_{2}.$$
Now let's turn our attention to $\Trace(A^{2}_{\alpha,\beta,\mu})$. The lower bounds on the diagonal terms only give that
$$a_{1}h^{-\alpha(n-1)+n-\beta}\leq{}\Trace(A^{2}_{\alpha,\mu}).$$
Therefore we need to use the off diagonal terms. What matters here is, how far apart do $\xi_{j}$ and $\xi_{m}$ have to be for
$$\int e^{\frac{i}{h}\langle y,\xi_{j}-\xi_{m}\rangle}(w^{\alpha,\mu}_{(x,\xi)}(y))^{2}dy$$
to be small? The additional smallness (apart from that which comes entirely from the support) is due to the non-stationary phase integral. If $\xi_{j}-\xi_{m}$ is large enough so that the oscillation of $e^{\frac{i}{h}\langle y,\xi_{j}-\xi_{m}\rangle}$ overwhelms the regularity of $w^{\alpha,\mu}_{(x,\xi)}(y)$ the contribution from this pair will be small. By rotations and translation it is enough to consider the case when $x=0$ and $\xi=e_{1}$. In that case $w^{\alpha,\mu}_{(0,e_{1})}$ is supported in a $h^{1-2\alpha}\times (h^{1-\alpha})^{n-1}$ tube (the long direction lies along the $e_{1}$ axis. First let's consider the Planck scale case (that is $\mu=1$). Let $\eta_{j,m}=\xi_{j}-\xi_{m}$. In this coordinate system if
\begin{equation}\begin{cases}
|(\eta_{j,m})_{1}|\leq\epsilon h^{2\alpha}\\
|(\eta_{j,m})'|\leq{}\epsilon h^{\alpha}\end{cases}\label{etacond}\end{equation}
the factor
$$e^{\frac{i}{h}\langle y,\eta_{j,m}\rangle}$$
does not complete a full oscillation over the support of $w^{\alpha,1}_{(x,\xi)}$. Therefore in these cases there are constants so that 
\begin{multline*}
a_{1}h^{-\alpha(n-1)}\chi\left(h^{-\alpha}\left|\frac{\xi_{j}}{|\xi_{j}|}-\xi\right|\right)\left(h^{-\alpha}\left|\frac{\xi_{m}}{|\xi_{m}|}-\xi\right|\right)\leq |I_{j,m}|\\
\leq a_{2}h^{-\alpha(n-1)}\chi\left(h^{-\alpha}\left|\frac{\xi_{j}}{|\xi_{j}|}-\xi\right|\right)\left(h^{-\alpha}\left|\frac{\xi_{m}}{|\xi_{m}|}-\xi\right|\right).\end{multline*}
Suppose that $j$ is one of the $\epsilon^{n-1}h^{\alpha(n-1)+\gamma-n}$ points such that 
$$\left|\frac{\xi_{j}}{|\xi_{j}|}-e_{1}\right|\leq \epsilon h^{\alpha}$$
and similarly $m$ is one of the points such that
$$\left|\frac{\xi_{m}}{|\xi_{m}|}-\xi\right|\leq{}\epsilon h^{\alpha}$$
that is both $\xi_{j}$ and $\xi_{m}$ lie in the intersection between the cone with angle $\frac{\epsilon h^{\alpha}}{2}$ from the $e_{1}$ axis and the annulus $[1-h^{\beta},1+h^{\beta}]$ (see Figure \ref{fig:coneann}).
\begin{figure}[h!]
\includegraphics[scale=0.4]{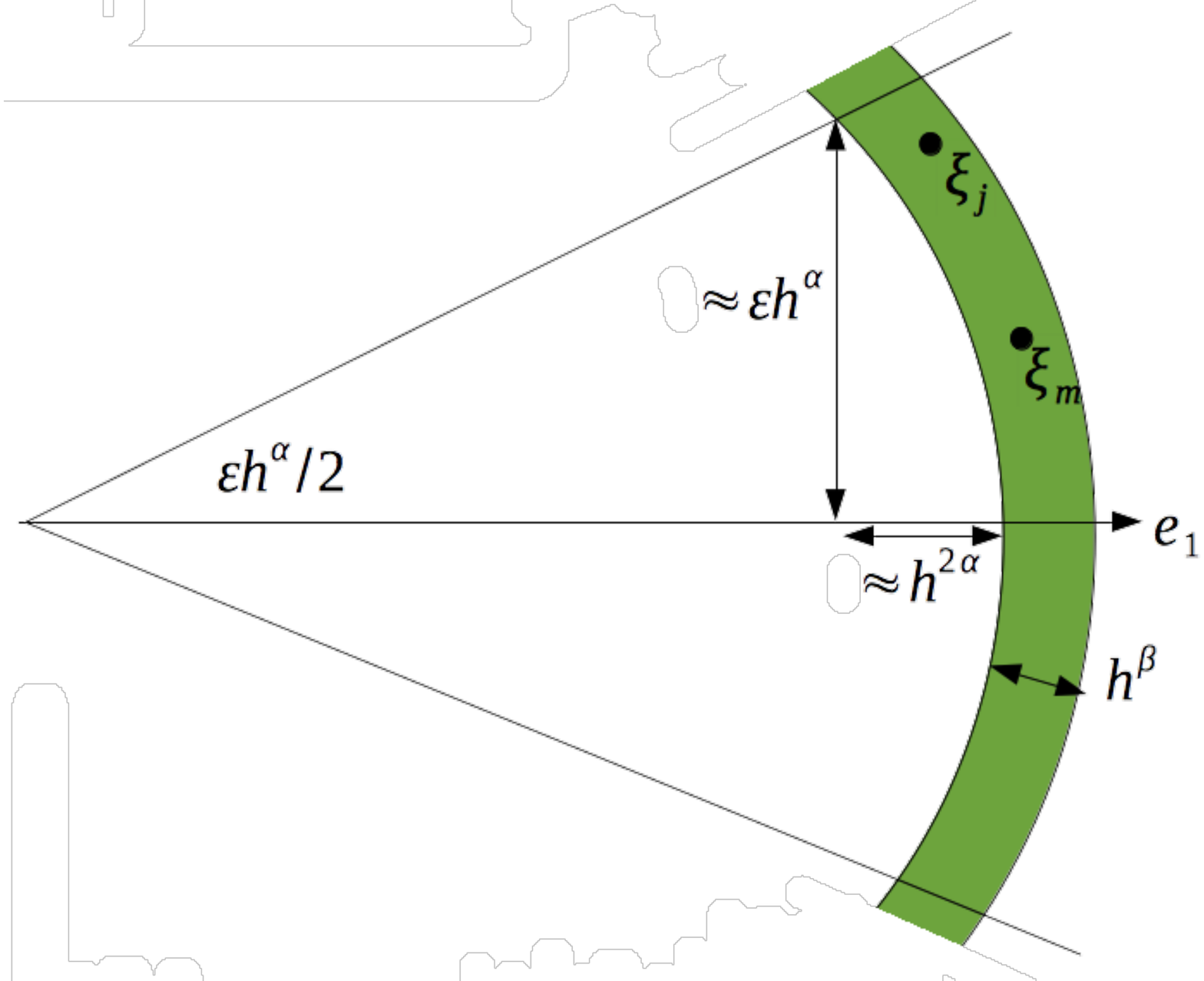}\label{fig:coneann}
\caption{If $\xi_{j}$ and $\xi_{m}$ are too close to each other the oscillations of $e^{\frac{i}{h}\langle y,\xi_{j}-\xi_{m}\rangle}$ are not enough to overwhelm the regularity of the symbol}
\end{figure}

 Therefore indeed, since $\alpha\leq{}\beta/2$, $\eta_{j,m}$ satisfies the conditions \ref{etacond}. So there are at least $\epsilon^{2(n-1)}h^{2\alpha(n-1)+2\gamma-2n}$ points for which
$$|I_{j,m}|>a_{\epsilon}h^{-\alpha(n-1)}$$
which yields
 $$\Trace(A_{\alpha,\beta,1})\geq{}a_{\epsilon}.$$
 In this case the upper bound follows directly from the maximum number of $(\xi_{j},\xi_{m})$ so the integrand of $I_{j,m}$ is nonzero.
 
 Now consider what happens as we move away from Planck scale $\mu\gg 1$. In this case the factor $e^{\frac{i}{h}\langle y,\eta_{j,m}\rangle}$ does not significantly oscillate if
 \begin{equation}\begin{cases}
 |(e_{j,m})_{1}|\leq \mu^{-2}h^{2\alpha}\\
 |(\eta_{j,m})'|\leq{}\mu^{-1}h^{\alpha}\end{cases}\label{etacondmu}
 \end{equation}
 but otherwise we are able to obtain some extra decay. Suppose first that 
 $$|(\eta_{j,m})_{1}|>2^{l}\mu^{-2}h^{\alpha}$$
 then integrating by parts in $y_{1}$ we find that
 $$\int e^{\frac{i}{h}\langle y,\eta_{j,m}\rangle}(w^{\alpha,\mu}(y))^{2}dy=\int\frac{he^{\frac{i}{h}\langle y,\eta_{j,m}\rangle}}{(\eta_{j,m})_{1}}\mu^{-2}h^{-1+2\alpha}\tilde{w}^{\alpha,\mu}(y)dy$$
 where $\tilde{w}^{\alpha,\mu}(y)$ has the same support and regularity properties as $w^{\alpha,\mu}$. Repeated applications of this argument show that for any $L$
\begin{equation}|I_{j,m}|\leq{}2^{-lL}h^{\alpha(n-1)}.\label{offdiag}\end{equation}
 If $|(\eta_{j,m})'|>\mu^{-1}h^{\alpha}$ the same argument using integration by parts in $y'$ variables gives \eqref{offdiag}. By picking $L$ large enough we can ensure that the major contribution to the sum $\sum_{j,m}|I_{j,m}|^{2}$ comes when \eqref{etacondmu} are satisfied. For any fixed $j$ there are $O(h^{\beta-n}h^{\alpha(n-1)}\mu^{-n-1}$) suitable $m$. Therefore we arrive at the estimate that
 $$a_{3}\mu^{-n-1}\leq \Trace(A^{2}_{\alpha,\beta,\mu})\leq{}a_{4}\mu^{-n-1}.$$

\end{proof}

The second lemma obtains Lipschitz bounds on $G=\norm{G_{\alpha,\beta,\mu}}_{L^{p}(B_{1}(0)\times \S^{n-1})}$. These are key to using the measure concentration arguments. 

\begin{lemma}\label{lem:Lipbnds}
Suppose $G_{p}=\norm{G_{\alpha,\beta,\mu}}_{L^{p}(B_{1}(0)\times \S^{n-1})}$ the following Lipschitz bound holds,
\begin{equation}|G_{p}(c)-G_{p}(d)|\leq N^{\frac{1}{2}-\frac{1}{p}}\mu^{-\frac{n+1}{4}+\frac{n+1}{2p}}\norm{c-d}_{\ell^{2}}=\mu^{-\frac{n+1}{4}+\frac{n+1}{2p}}h^{\frac{\beta-n}{2}\left(1-\frac{1}{p}\right)}\norm{c-d}_{\ell^{2}}.\label{Glip}\end{equation}

\end{lemma}

\begin{proof}
Let 
\begin{align*}
u&=\sum_{\xi_{j}\in\Lambda_{\beta}}c_{j}e^{\frac{i}{h}\langle x,\xi_{j}\rangle}& \quad &v=\sum_{\xi_{j}\in \Lambda_{\beta}} d_{j}e^{\frac{i}{h}\langle x,\xi_{j}\rangle}\\
U(x,\xi)&=\norm{ P^{\alpha,\mu}_{(x,\xi)}u(\cdot)}_{L^{2}} & \quad &V(x,\xi)=\norm{ P^{\alpha,\mu}_{(x,\xi)}v(\cdot)}_{L^{2}}.\end{align*}
In this notation
$$|G_{p}(c)-G_{p}(d)|=\left|\norm{U}_{L^{p}(B_{1}(0)\times S^{n-1})}-\norm{V}_{L^{p}(B_{1}(0)\times S^{n-1})}\right|\leq \norm{U-V}_{L^{p}(B_{1}(0)\times\S^{n-1})}.$$
So if we can find estimates for $p=\infty$ and $p=2$ we can interpolate all the others. First let's see the $p=\infty$ case. 
$$U(x,\xi)-V(x,\xi)=\norm{P^{\alpha,\mu}_{(x,\xi)}u(\cdot)}_{L^{2}}-\norm{P^{\alpha,\mu}_{(x,\xi)}v(\cdot)}_{L^{2}}\leq\norm{P^{\alpha,\mu}_{(x,\xi)}(u-v)(\cdot)}_{L^{2}}.$$
$$\norm{P^{\alpha,\mu}_{(x,\xi)}(u-v)(\cdot)}_{L^{2}}=\left(\sum_{j,m}(c_{j}-d_{j})(c_{m}-d_{m})I_{j,m}\right)^{1/2},$$
where $I_{j,m}$ are as in the proof of Lemma \ref{lem:tracebnds}. Applying Cauchy-Schwartz 
$$\norm{P^{\alpha,\mu}_{(x,\xi)}(u-v)(\cdot)}_{L^{2}}\leq\norm{c-d}_{\ell^{2}}\left(\sum_{j,m} I_{j,m}^{2}\right)^{1/4}.$$
In the proof of Lemma \ref{lem:tracebnds} we have already estimated the sum $\sum_{j,m}|I_{j,m}|^{2}$ by
$$\sum_{j,m}|I_{j,m}|^{2}\leq C\mu^{-n-1}N^{2}.$$
So this leads to a Lipschitz bound
$$|G_{\infty}(c)-G_{\infty}(d)|\leq C\mu^{-\frac{n+1}{4}}N^{1/2}=C\mu^{-\frac{n+1}{4}}h^{\frac{\beta-n}{2}}.$$
Now for the $p=2$ case. In that case
\begin{align*}
\norm{U-V}_{L^{2}}&=\int_{B_{1}(0)\times \S^{n-1}}\langle P^{\alpha,\mu}_{(x,\xi)}(u-v),P^{\alpha,\mu}_{(x,\xi)}(u-v)\rangle_{y}dxd\omega(\xi)\\
&=\sum_{j,m}(c_{j}-d_{j})(c_{m}-d_{m})\int_{y}\int_{B_{1}(0)\times \S^{n-1}}\left(P^{\alpha,\mu}_{(x,\xi)}e^{\frac{i}{h}\langle \cdot,\xi_{j}\rangle}\overline{P^{\alpha,\mu}_{(x,\xi)}e^{\frac{i}{h}\langle \cdot,\xi_{m}\rangle}}\right)\Big|_{y}dxd\omega(\xi) dy\\
&=\sum_{j,m}(c_{j}-c_{m})(c_{m}-d_{m})\int e^{\frac{i}{h}\langle y,\xi_{j}-\xi_{m}\rangle}a(y)dy.\end{align*}
To get the last line we have integrated the $(x,\xi)$ variables first, integrating out the cut-off functions. Therefore since the $\xi_{j}$ are spaced at order $h$ (and are therefore almost orthogonal) we have
$$\norm{U-V}_{L^{2}}\leq C\norm{c-d}_{\ell^{2}}.$$
Interpolating with $L^{\infty}$
$$\norm{U-V}_{L^{p}}\leq C N^{\frac{1}{2}-\frac{1}{p}}\mu^{-\frac{n+1}{4}+\frac{n+1}{2p}}\norm{c-d}_{\ell^{2}}$$
which establishes the Lipschitz bounds.

\end{proof}

\section*{Acknowledgments} The author would like to thank Alex Barnett and Xiaolong Han for many interesting discussions on the small-scale structure of random waves. Particular thanks to Alex Barnett for allowing the reproduction of his numerical studies in this paper. 

\bibliography{references}
\bibliographystyle{plain}

\end{document}